\documentclass[12pt,a4paper, reqno]{amsart}
\usepackage{amsfonts,amsmath,amsthm,amssymb, amscd}
\usepackage[top=2cm, bottom=5cm, left=3cm, right=1cm]{geometry}

\usepackage{color}

\newcommand{\tifrac}{\text{\raisebox{-4,5pt}[0pt][0pt]{$\widetilde{\phantom{nn}}\kern-1.1em$}\,} \genfrac{}{}{0pt}{}}
\newcommand{\tifracc}{\text{\raisebox{-6,5pt}[0pt][0pt]{$\widetilde{\phantom{nn}}\kern-1.1em$}\,} \genfrac{}{}{0pt}{}}
\newcommand{\tifraccc}{\text{\raisebox{-6,5pt}[0pt][0pt]{$\widetilde{\phantom{nnn}}\kern-1.5em$}\,} \genfrac{}{}{0pt}{}}
\newcommand{\tifracccc}{\text{\raisebox{-7pt}[0pt][0pt]{$\widetilde{\phantom{mmn}}\kern-2.1em$}\,} \genfrac{}{}{0pt}{}}

\newtheorem{theorem}{Theorem}
\newtheorem{lemma}{Lemma}
\newtheorem{corollary}{Corollary}

\def\gp#1{\langle #1 \rangle}

\title[Elements of the Zelisko group II]{Generating solutions of a linear equation   and
structure of elements of the Zelisko group II}


\author[Bovdi and Shchedryk]{V.A.~Bovdi and V.P.~Shchedryk}
\address{United Arab Emirates University, Al Ain, UAE}
\email{vbovdi@gmail.com}
\address{Pidstryhach Institute for Applied Problems of Mechanics
and Mathematics,  National Academy of Sciences of Ukraine, Lviv,  Ukraine}

\email{shchedrykv@ukr.net}

\begin{document}

\maketitle

\begin{abstract}
We  continue our  previous investigation of the Zelisko group of a matrix over B\'ezout domains.
The explicit form of elements of this group over homomorphic image of B\'ezout domain of stable rank 1.5 is described.
\end{abstract}

\section{Introduction }

Let $R$ be a commutative B\'ezout domain (finitely generated principal ideal domain) with $1\not=0$  and let $R^{n\times n}$ be the ring  of ${n\times n}$ matrices over $R$ in which $n\geq 2$. Let $U(R)$ and ${\rm GL}_n(R)$ be  groups of units of  rings $R$ and $R^{n\times n}$, respectively.
The notation $a| b$ in $R$ means that  $b=ac$ for some $c\in R$. The greatest common divisor of $a,b\in R$ is denoted  by $(a,b)$.

To  a diagonal matrix
$\Phi:={\rm diag}(\varphi_{1}, \ldots, \varphi_{k}, 0, \ldots, 0)\in R^{n\times n}$,
where $\varphi_{k} \ne 0$, $k \leq n$, and  $\varphi_{i}$ is a divisor of $\varphi_{i+1}$ for  $i=1, \ldots, k-1$,  we  associate   the  following subgroup (see \cite[p.\,61]{Shch_Mon} and \cite{Kaz2, Zel})
\[
{\bf{G}}_{\Phi } =\{ H\in{\rm GL}_{n} (R)\mid \; \exists S\in{\rm GL}_{n} (R), \;\text{s.t.}\;  H\Phi =\Phi S\; \} \leq{\rm GL}_{n} (R),
\]
 which is called the {\it Zelisko group} of the matrix $\Phi$.
The concept of the Zelisko group as well as its properties,   were  used by Kazimirski\u{\i} \cite{Kaz2} for the  solution of the problem of extraction of a regular divisor  of a  matrix over the  polynomial ring $F[x]$, where  $F$ is an algebraically closed  field of characteristic $0$. The properties of  the group ${\bf{G}}_{\Phi }$ in which  $\Phi\in R^{n\times n}$,   were explicitly investigated in \cite[Chapter 2.3 and  Chapter 2.7]{Shch_Mon} and \cite{Bovdi_Shchedryk_II, Zel}.

A ring $K$ has {\it  stable range} $1.5$ (see \cite[p.\,961]{Shch_StablRange} and  \cite[p.\,46]{Shchedryk_3})  if for each $a, b\in K$ and  $c \in K\setminus\{0\}$   with the property  $(a,b,c) = 1$ there exists $r \in K$ such that
$(a + br, c) = 1$.
This notion  arose as  a modification of the Bass's concept of the stable range of  rings (see \cite[p.\,498]{Bass}). Examples of rings of stable range $1.5$  are Euclidean rings,  principal ideal rings,  factorial rings, rings of algebraic integers,
rings of integer analytic functions, and  adequate  rings (see \cite{Bovdi_Shchedryk_I, Bovdi_Shchedryk_II} and \cite[p.\,21]{Shch_Mon}). Note that the commutative rings of stable range $1.5$ coincide with rings of almost stable range $1$ (see \cite{Anderson_Juett, McGovern}).
Finally, certain  properties of  the Zelisko group  ${\bf{G}}_{\Phi }$  are  closely related to a  factorizability of the general linear group over the ring $R$ of stable range $1.5$ (see \cite[Theorem 3, p.\,144]{Shch_DecompGroup}, \cite[Chapter 2.7]{Shch_Mon},  \cite[Theorem 1.2.2.,  p.\,12]{Chen},   \cite{Chen_2}, and \cite{Petechuk_Petechuk}).

Let $R$ be    a commutative B\'ezout domain   of stable range 1.5.
 For each $m\in R\setminus\{0, U(R)\}$ we define  the homomorphism $\overline{\bullet}: R\to R_m=R/mR$.

The explicit appearance of the elements of the group ${\bf{G}}_{\Phi}$ in a particular  case when $\Phi: ={\rm diag} (\varphi_{1}, \ldots, \varphi_{n}) \in R_m ^{n \times n}$ in which  $ \varphi_{n} \neq 0$ and   $\varphi_{i}$ is a divisor of $\varphi_{i+1}$ for  $i=1, \ldots, n-1$ was studied  in \cite{Bovdi_Shchedryk_II}. In the present  article, we continue  to investigate the form of elements of the Zelisko group  in general case, when  some diagonal elements of the matrix $\Phi$ can be zero.

First of all, we recall some definitions and facts from \cite{Bovdi_Shchedryk_II}.

The  solution  of a solvable linear equation $\overline{a}\cdot \overline{x}=\overline{b}$ in $R_m$  which divides all other solutions is called  {\it generating solution}  of this equation.

Each   solvable linear equation  $\overline{a}\cdot \overline{x}=\overline{b}$ in  $R_m$ has   at least one  generating solution by  \cite[Theorem 1(i)]{Bovdi_Shchedryk_II} and   each two  generating solutions of this  equation  are pairwise associates by  \cite[Theorem 1(ii)]{Bovdi_Shchedryk_II}.

 For each $c\in R$, let  $\overline{c}:=\overline{\bullet}(c)\in R_m$. Using \cite[Lemma 3]{Bovdi_Shchedryk_II}, we have
\begin{equation}\label{EQV:1}
\overline{c}=\overline{\mu}_{c}\;\overline{e},
\end{equation}
where  ${\mu_{c} }:=(c, m)$ is a preimage of $\overline{\mu}_{c}$ and $\overline{e} \in U(R_m)$.  However,  such representation of the element $\overline{c}$ is not unique (see Example 1 after \cite[Lemma 3]{Bovdi_Shchedryk_II}). The explicit form of a solution of the linear equation $\overline {a} \cdot \overline {x} = \overline {b}$ in $ R_m $ depends on the choice of the representation of the elements $\overline {a}, \overline {b} \in  R_m $ in the form \eqref{EQV:1} (see proof of \cite [Theorem 1(i)]{Bovdi_Shchedryk_II}). The next  example illustrates the relationship between the solutions of the linear equation  $\overline {a} \cdot \overline {x} = \overline {b}$ in $ R_m $ obtained with different representation of elements $\overline {a}, \overline {b}$ in the form \eqref{EQV:1}.

\smallskip
\noindent{\bf Example 1}.
Let  $R_{m} =\mathbb{Z}_{36}$. Consider  $\overline{33}\cdot\overline{x}=\overline{30}$.
Clearly
$\overline{33}=\overline{3}\cdot\overline{11}$ and $\overline{30}=\overline{6}\cdot\overline{5}$ in which
$(33, 36)=3$, $(30, 36)=6$, and $\overline{11}, \overline{5}\in U(\mathbb{Z}_{36})$.
A generating  solution of the equation $\overline{33} \cdot \overline{x} = \overline{30}$   is
\[
\textstyle
\overline{x}_0=\overline{\left(\frac{6}{3}\right)}\cdot \overline{5}\cdot (\overline{11})^{-1}=
\overline{2}\cdot \overline{5}\cdot \overline{23}=\overline{14},
\]
where
$\overline{23}=(\overline{11})^{-1}$ (see \cite [Theorem 1(i)]{Bovdi_Shchedryk_II}). Since $Ann (\overline{33}) = \gp{\overline{12}}$,  we conclude that $\overline{14}+Ann(\overline{33})=\{\overline{2}, \overline{14}, \overline{26}\}$ is the set of all solutions of
$\overline{33} \cdot \overline{x} = \overline{30}$.

Let's now choose another representation  of  $\overline{33}$ and $\overline{30}$ in the form (\ref{EQV:1}):
\[
\begin{split}
\overline{33}&=\overline{3}\cdot\overline{35},\; \text{where}\;(33, 36)=3,\; \text{and}\; \overline{35}\in U(\mathbb{Z}_{36}),\\\overline{30}&=\overline{6}\cdot\overline{35},\; \text{where}\;(30, 36)=6,\; \text{and}\; \overline{35}\in U(\mathbb{Z}_{36}).
\end{split}
\]
Clearly, the generating solution is $\overline{x}_0 = \overline{2}$.
Thus, we obtain  another generating solution, which by virtue of \cite [Theorem 1(ii)]{Bovdi_Shchedryk_II} are associated with each other.

\section{Preliminaries}

We start our proof with the following.

\begin{lemma}\label{L:1} Let  $\overline{c}_1, \overline{c}_2, \overline{c}_3$ be nonzero elements of  $R_m$ such that
$\overline{c}_1 \;| \; \overline{c}_2 \; |\;  \overline{c}_3$. Let
$\overline{\gamma}_{21}, \overline{\gamma}_{32}$ be arbitrary generating solutions of
$\overline{c}_1\cdot \overline{ x}=\overline{c}_2$  and
$\overline{c}_2\cdot \overline{ x}=\overline{c}_3$, respectively. Then
$\overline{\gamma}_{21}\cdot\overline{\gamma}_{32}$ is a generating solution
of $\overline{c}_1\cdot \overline{ x}=\overline{c}_3$.
\end{lemma}

\begin{proof}
If $\overline{c}_1 \;| \; \overline{c}_2 \; |\;  \overline{c}_3 $, then  there  exist  ${c}_1, \;{c}_2,\;{c}_3  \in R$ such that  ${c}_1 \;| \;{c}_2 \; |\;{c}_3 $  by \cite[Lemma 1]{Bovdi_Shchedryk_II}.
Moreover   $\overline{c}_i=\overline{\mu}_{c_i}\;\overline{e}_{i}$, where  ${\mu_{c_i} }:=(b_i, m)$ is a preimage of $\overline{\mu}_{c_i}$ and $\overline{e}_{i} \in U(R_m)$ for  $i=1,2,3$ by \cite[Lemma 3]{Bovdi_Shchedryk_II}.
Now by analogy to proof of  \cite[Theorem 1(i)]{Bovdi_Shchedryk_II}  we have
\[
\textstyle
\overline{\psi}_{ij}=\overline{\left(\frac{\mu_{c_i}}{\mu_{c_j} }\right)}\overline{e}_{i}
(\overline{e}_{j})^{-1},\qquad (1\leq j< i \leq 3)
\]
is a generating solution of the equation $\overline{c}_j\cdot \overline{ x}=\overline{c}_i$.
From the property  $\mu_{c_1}\;| \; \mu_{c_2}\;| \; \mu_{c_3}$, follows that
$
\frac{\mu_{c_2}}{\mu_{c_1} } \cdot \frac{\mu_{c_3}}{\mu_{c_2} }=\frac{\mu_{c_3}}{\mu_{c_1} }$ and
\[
\textstyle
\overline{\left(\frac{\mu_{c_2}}{\mu_{c_1} }\right)}\cdot
\overline{\left(\frac{\mu_{c_3}}{\mu_{c_2} }\right)}=
\overline{\left(\frac{\mu_{c_3}}{\mu_{c_1} }\right)}.
\]
Consequently, we have
\[
\begin{split}
\textstyle
\overline{\psi}_{21}\cdot \overline{\psi}_{32}&=
\textstyle\overline{\left(\frac{\mu_{c_2}}{\mu_{c_1} }\right)}\overline{e}_{2}
(\overline{e}_{1})^{-1}\cdot
\overline{\left(\frac{\mu_{c_3}}{\mu_{c_2} }\right)}\overline{e}_{3}
(\overline{e}_{2})^{-1}\\
&=\textstyle\overline{\left(\frac{\mu_{c_2}}{\mu_{c_1} }\right)}\cdot
\overline{\left(\frac{\mu_{c_3}}{\mu_{c_2} }\right)}
\overline{e}_{3}
(\overline{e}_{1})^{-1}
\overline{e}_{2}
(\overline{e}_{2})^{-1}\\
&=
\textstyle\overline{\left(\frac{\mu_{c_3}}{\mu_{c_1} }\right)}
\overline{e}_{3}
(\overline{e}_{1})^{-1}\\&=
\overline{\psi}_{31}.
\end{split}
\]
\end{proof}
\noindent
Let $\overline{\varphi}_1, \overline{\varphi}_2, \ldots,  \overline{\varphi}_t \in \{R_m \setminus \{0\}\}$, such that  $\overline{\varphi}_1 \;| \; \overline{\varphi}_2 \; |\; \cdots \;| \; \overline{\varphi}_t$.
Denote by $\overline{M}_{ij}$ the set of all solutions of the equation
$\overline{\varphi}_j \cdot \overline{x} = \overline{\varphi}_i$, where $i> j$. By  symbol
$\textstyle \displaystyle{\tifracc{{\overline{\varphi}}_i}{{\overline{\varphi}}_j}}$
we denoted the minimum generating solution  from  $\overline{M}_{ij}$ with respect to some selected relation of order $\leq$ (see  \cite{Bovdi_Shchedryk_II}). Such choice is not good enough. Therefore, we propose the following improvement of this notation.

In each set $\overline{M}_{i + 1, i}$ we fix a generating  solution  denoted
\[
\textstyle{\tifracccc{\overline{\varphi}_{i + 1}}{\overline{\varphi}_{i}}} \qquad    (i = 1, \ldots, n-1).
\]
Fix integers $p$ and $q$ such that $p>q+1\geq 2$.   The product
\begin{equation}\label{EQV:2}
\textstyle\tifracccc{\overline{\varphi}_{q+1} }{\overline{\varphi}_{q}}\cdot
\tifracccc{\overline{\varphi}_{q+2} }{\overline{\varphi}_{q+1}} \ldots \tifracccc{\overline{\varphi}_{p-1} }{\overline{\varphi}_{p-2}}  \cdot  \tifracccc{\overline{\varphi}_{p} }{\overline{\varphi}_{p-1}}
\end{equation}
of generating solutions of sets $\overline{M}_{q + 1, q}$, $\overline{M}_{q + 2, q+1}$, \ldots, $\overline{M}_{p-1, p-2}$, $\overline{M}_{p, p-1}$, respectively,  belongs to $\overline{M}_{p q}$ by Lemma \ref{L:1} and we denote it by ${\tifracc{{\overline{\varphi}}_p}{{\overline{\varphi}}_q}}$.

Let   $\overline{c}_1,\overline{c}_2 \in \{R_m \setminus \{0\}\}$. Then
$\overline{c}_i=\overline{\mu}_{{c}_i}\overline{e}_i$, where
${\mu}_{c_i}=({c}_i, m)$ and   $\overline{e}_i \in U(R_m)$ for   $i=1,2$ by \eqref{EQV:1}.
Since ${\rm Ann}(\overline{c}_i)=\overline{\alpha}_{c_i}R_{m}$, where
\begin{equation}\label{EQV:3}
\textstyle
\alpha_{c_i}:=\frac{m}{\mu_{c_i}}
\end{equation}
 and $\mu_{c_i}:=(b_i,m)$ by \cite[Lemma 5]{Bovdi_Shchedryk_II}, we set $\overline{\alpha}_i:=\overline{\alpha}_{c_i}\overline{e}_i$. Hence
\[
\overline{\alpha}_iR_{m}=(\overline{\alpha}_{c_i}\overline{e}_i)R_{m}=\overline{\alpha}_{c_i}
(\overline{e}_iR_{m})=\overline{\alpha}_{c_i}R_{m}={\rm Ann}(\overline{c}_i).
\]
Note that if $\overline{c}_1 \;| \; \overline{c}_2$, then  we define $\tifracc{\overline{{c}}_2}{{\overline{c}}_1}$ as above.

\begin{lemma}\label{L:2}
Let   $\overline{c}_1, \overline{c}_2 \in \{R_m \setminus \{0\}\}$ such that
$\overline{c}_1 \;| \; \overline{c}_2$. Then $\tifracc{\overline{{c}}_2}{{\overline{c}}_1}$ is a generating solution of the equation
\begin{equation}\label{EQV:4}
\overline{\alpha}_2\overline{x}=\overline{\alpha}_1,
\end{equation}
where each $\gp{\overline{\alpha}_{i}}_{R_{m}}={\rm Ann}(\overline{c}_i)$ and  $\alpha_i$ has form \eqref{EQV:3} for $i=1,2$.
\end{lemma}

\begin{proof}
Since  $\overline{c}_2=\overline{c}_1\overline{d}$, where  $\overline{d} \in R_m$,    there are exist   $c_1, c_2, d \in R$ such that  $c_2= c_1 d$
by  \cite[Lemma 1]{Bovdi_Shchedryk_II}. This yields that  ${\mu}_{c_1}:=({c}_1, m) | ({c}_2, m)={\mu}_{c_2}$ and
\[
\begin{split}
\alpha_1=\alpha_{c_1}e_1&=\textstyle\frac{m}{\mu_{c_1}}e_1=\frac{m}{\mu_{c_2}}\frac{\mu_{c_2}}{\mu_{c_1}}e_1=\\&
=\textstyle\alpha_{c_2}\frac{\mu_{c_2}}{\mu_{c_1}}e_1=(\alpha_{c_2}e_2)\frac{\mu_{c_2}}{\mu_{c_1}}e_1e_2^{-1}=
\alpha_2\sigma,
\end{split}
\]
{%\color{red}
in which we define
%\begin{equation}\label{NGYY}
$\sigma:=\textstyle \frac{\mu_{c_2}}{\mu_{c_1}}e_1e_2^{-1}$.
%\end{equation}
}
Thus    $\overline{\alpha}_2 \;\overline{\sigma}=\overline{\alpha}_1$,
so   $\overline{\sigma}+Ann(\overline{\alpha}_2)$ is the set of all solutions of \eqref{EQV:4}.

Noting that the preimages   of the invertible elements of the ring $ R_m $ are those elements of $R$ that are relatively  prime with $ m $, we obtain that the preimage $\overline{\alpha}_2 $ in $ R $ is the element $\alpha_{2} \varepsilon_2 $ where $(\varepsilon_2, m) = 1$.
According to \cite[Lemma 5]{Bovdi_Shchedryk_II} $Ann(\overline{\alpha}_2)=\overline{a}_2R_m,$ where
\[
\textstyle
{a}_2=\frac{m}{(\alpha_{2}\varepsilon, m)}=\frac{m}{(\alpha_{2}, m)}=
\frac{m}{\left(\frac{m}{\mu_{c_2}}, m \right)}=\frac{m}{\frac{m}{\mu_{c_2}}}=\mu_{c_2}.
\]
Since $\sigma= \frac{\mu_{c_2}}{\mu_{c_1}}e_1e_2^{-1}$, we get
$\mu_{c_2}=\sigma \mu_{c_1}e_1^{-1}e_2$ and   $\overline{\sigma}|\overline{\mu}_{c_2}$. Therefore $\overline{\sigma}$ is the divisor of all elements
of $Ann(\overline{\alpha}_2) $, so it  is a divisor of all  solutions  of \eqref{EQV:4} and  $\overline{\sigma}$ is a generating solution.
\end{proof}

{%\color{red}
Using the notation of Lemma \ref{L:2} we put $\textstyle
\tifracc{\overline{\alpha}_1}{\overline{\alpha}_2}:=\tifracc{\overline{{c}}_2}{{\overline{c}}_1}$.

\begin{corollary}\label{Co:1}
Let   $\overline{c}_1, \overline{c}_2 \in \{R_m \setminus \{0\}\}$ such that
$\overline{c}_1 \;| \; \overline{c}_2$.
Then $\tifracc{\overline{{c}}_2}{{\overline{c}}_1}\cdot{{\overline{\alpha}}_2}=
{{\overline{\alpha}}_1}$.
\end{corollary}

}

\begin{proof} Clearly,
$ \tifracc{\overline{{c}}_2}{{\overline{c}}_1}\cdot\overline{\alpha}_2=
\tifracc{\overline{\alpha}_1}{{\overline{\alpha}}_2}\overline{\alpha}_2$.
The element $\tifracc{\overline{\alpha}_1}{{\overline{\alpha}}_2}$
is a solution of
$\overline{\alpha}_2 \cdot \overline{x}  =   \overline{\alpha}_1$, so
$\tifracc{\overline{\alpha}_1}{{\overline{\alpha}}_2}\cdot \overline{\alpha}_2=\overline{\alpha}_1$ and
 $\tifracc{\overline{{c}}_2}{{\overline{c}}_1}\cdot{{\overline{\alpha}}_2}=
{{\overline{\alpha}}_1}$.
\end{proof}

Note that  $ \tifracc{\overline{{c}}_2}{{\overline{c}}_1}$ and
$ \tifracc{\overline{\alpha}_1}{{\overline{\alpha}}_2}$
are  generating solutions  of $ \overline{{c}}_1 \cdot \overline{x} = \overline{c}_2 $ and
$ \overline{\alpha}_2 \cdot \overline{x} = \overline{\alpha}_1 $,  respectively.
It should be noted that  sets of solutions  of these equations, generally speaking, do not coincide. Moreover, if $\overline{\alpha}_1, \overline{\alpha}_2$ have no form \eqref{EQV:3}, then $ \tifracc{\overline{{c}}_2}{{\overline{c}}_1}$ may not be a solution  of the equation $\overline{\alpha}_2 \cdot \overline{x} = \overline{\alpha}_1$.

\noindent{\bf Example 2}.
Let  $R_{m} =\mathbb{Z}_{144}$. Consider the equation
$\overline{4} \,\overline{x} = \overline{8}$, in which  $\overline{{c}}_1=\overline{4}$ and $\overline{{c}}_2=\overline{8}$. Since
$Ann (\overline{4}) = \{\overline{0}, \overline{36}, \overline{72} \} = \overline{36} \mathbb{Z}_{144}$, set of
solutions of this equation is $\overline{2}+Ann (\overline{4})=\{\overline{2}, \overline{38}, \overline{74} \}$,
and  all of those solutions is  generating. Clearly,
\[
Ann(\overline{8})= \{\overline{0}, \overline{18}, \overline{36},
\overline{54}, \overline{72}, \overline{90}, \overline{108},  \overline{126}
\}=\overline{18} \mathbb{Z}_{144}=\overline{126} \mathbb{Z}_{144}.
\]
Set $\overline{\alpha}_1= \overline{36}$ and $\overline{\alpha}_2 =\overline{18}$.
All solutions of the equation $ \overline{18} \overline{x} = \overline{36}$ belong to
\[
S=\{\overline{2}, \overline{10}, \overline{18},
\overline{26}, \overline{34}, \overline{42}, \ldots,  \overline{74}, \ldots,   \overline{138}
\}.
\]
Hence, if
$\tifraccc{\overline{8}}{\overline{4}}\;:=\overline{38}$, then $\tifraccc{\overline{8}}{\overline{4}}\not\in S$.

On the other hand, if we put $\overline{\alpha}_1= \overline{36}$ and
$\overline{\alpha}_2 =\overline{126}$, then  $\overline{38}$ is a solution  of the equation
$\overline{126}\overline{x}=\overline{36}$.

\begin{lemma}\label{Lemma:3} Let $R$ be a commutative ring and let
\[
A=
\left[\begin{matrix}
1 & 1 & 1 &{\ldots } & 1 & 1 & 1 \\a_{21} & 1 & 1 &{\ldots } & 1 & 1 & 1 \\b_{31} & a_{32} & 1 &{\ldots } & 1 & 1 & 1 \\{\ldots } &{\ldots } &{\ldots } &{\ldots } &{\ldots } &{\ldots } &{\ldots } \\b_{n1} & b_{n2} &  b_{n3} &{\ldots } &  b_{n,n-2} & a_{n,n-1} & 1
\end{matrix}\right]\in R^{n\times n},
\]
in which\quad  $b_{ij}=a_{j+1,j} a_{j+2,j+1} \cdots a_{i-1,i-2}a_{i,i-1}$,
where $i>j+1$,  $i=3, \ldots, n$,  and   $j=1, \ldots, n-2$.

If ${\det}(A)=\sum_{\sigma \in S_n}\gamma_{\sigma}$ and ${\det}(A^T)=\sum_{\sigma \in S_n}\delta_{\sigma}$ are classical  decompositions  by definition of ${\det}(A)$ and ${\det}(A^T)$ {%\color{red}
into terms,} respectively, then
\[
\gamma_{\sigma}= \gamma_{{\sigma}^{-1}}= \delta_{{\sigma}}.
\]

% following conditions hold:
%\begin{itemize}
%\item[\rm{(i)}] $\gamma_{\sigma}= \gamma_{{\sigma}^{-1}}$;
%\item[\rm{(ii)}]  $\gamma_{\sigma}= \delta_{{\sigma}}$.
%\end{itemize}
\end{lemma}

\begin{proof}
To each $\sigma=\left(\begin{smallmatrix}{1} &{2} &{\ldots } &{n} \\{i_{1} } &{i_{2} } &{\ldots } &{i_{n} } \end{smallmatrix}\right)\in S_n$  we assign  the following two sets:
\[
\begin{split}
\mathfrak{I}_1(\sigma)&=\textstyle\{ \;
(p_{i}, q_{p_{i}})  \; \mid  \;  p_{i} >q_{p_{i}} \quad \text{and}   \; \binom{ p_{i}}{q_{p_{i}}}\} \; \text{is a column in }\; \sigma,\; i= 1, \ldots,  s \};\\\mathfrak{I}_2(\sigma)&=\textstyle\{ \;
(\alpha_{i}, \beta_{\alpha_{i}})  \;\mid \;   \alpha_{i} \leq \beta_{\alpha_{i}} \;\; \text{and}  \; \binom{ \alpha_{i}}{\beta_{\alpha_{i}}} \; \text{is a column in }\;  \sigma, \; i= 1, \ldots, t\},
\end{split}
\]
where $s+t=n$.  If $A:=[c_{ij}]$, then in the decomposition of $\det(A)$  we have
\[
\begin{split}
\gamma_{\sigma}&=(-1)^{sign(\sigma)}c_{p_{1}, q_{p_{1}}} \cdot c_{p_{2}, q_{p_{2}}} \cdots c_{p_{s}, q_{p_{s}}};\\
\gamma_{\sigma^{-1}}&=(-1)^{sign(\sigma^{-1})}c_{\alpha_{1}, \beta_{\alpha_{1}}} \cdot  c_{\alpha_{2}, \beta_{\alpha_{2}}} \cdots c_{\alpha_{t}, \beta_{\alpha_{t}}},
\end{split}
\]
in which  $c_{p_{i}, q_{p_{i}}}$, as well as $c_{\alpha_{i}, \beta_{\alpha_{i}}}$ is a  product of elements of the first subdiagonal  of  $A$.
According to  \cite[Lemma 9]{Bovdi_Shchedryk_II} we obtain that
\[
\textstyle
\prod_{(p_i,q_{p_{i}})\in \mathfrak{I}_1(\sigma)}
 \frac{p_{i} }{q_{p_{i}}}=
\prod_{(\alpha_i, \beta_{\alpha_i})\in \mathfrak{I}_2(\sigma)}
\frac{\beta_{\alpha_i} }{\alpha_{i}},
\]
and direct calculation gives    $\gamma_{\sigma}= \gamma_{{\sigma}^{-1}}$.

Finally,  every term $\gamma_{\sigma}$ appearing in ${\det}(A)$ corresponds to the term $\gamma_{{\sigma} ^{- 1}}$ in ${\det} (A ^{T})$, so  $\gamma_{{\sigma} ^{- 1}}=\delta_{{\sigma}}$, as above.
\end{proof}

\begin{lemma}\label{Le:4} Let $A,B\in R^{n\times n}$ such that
\[
A=
\left[\begin{matrix}
h_{11} & h_{12} & h_{13} &{\ldots } &  h_{1,n-2} & h_{1,n-1} & h_{1n} \\a_{21}h_{11} & h_{22} & h_{23} &{\ldots } &  h_{2,n-2} & h_{2,n-1} & h_{2n} \\b_{31}h_{31} &a_{32} h_{32} & h_{33} &{\ldots } &  h_{3,n-2} & h_{3,n-1} & h_{3n} \\{\ldots } &{\ldots } &{\ldots } &{\ldots } &{\ldots } &{\ldots } &{\ldots } \\b_{n1}h_{n1} &b_{n2}h_{n2} & b_{n3}h_{n3} &{\ldots } &  b_{n,n-2}h_{n,n-2} & a_{n,n-1}h_{n,n-1} & h_{nn}
\end{matrix}\right]
\]
and
\[
B=
\left[\begin{matrix}
h_{11} & a_{21}h_{12} & b_{31}h_{13} &{\ldots } &  b_{n-2,1}h_{1,n-2} & b_{n-1,1}h_{1,n-1} & b_{n1}h_{1n} \\h_{11} & h_{22} & a_{32}h_{23} &{\ldots } &  b_{n-2,2}h_{2,n-2} & b_{n-1,2}h_{2,n-1} & b_{n2}h_{2n} \\h_{31} &h_{32} & h_{33} &{\ldots } &  b_{n-2,3}h_{3,n-2} & b_{n-1,3}h_{3,n-1} & b_{n3}h_{3n} \\{\ldots } &{\ldots } &{\ldots } &{\ldots } &{\ldots } &{\ldots } &{\ldots } \\h_{n-1,1} &h_{n-1,2} & h_{n-1,3} &{\ldots } &  h_{n-1,n-2} & h_{n-1,n-1} & a_{n,n-1}h_{n-1,n}\\h_{n1} &h_{n2} & h_{n3} &{\ldots } &  h_{n,n-2} & h_{n,n-1} & h_{nn}
\end{matrix}\right],
\]
in which $b_{ij}=a_{j+1,j} a_{j+2,j+1} \cdots a_{i-1,i-2}a_{i,i-1}$,
where $i>j+1$,  $i=3, \ldots, n$,  and   $j=1, \ldots, n-2$.

Then    the corresponding terms of  $\det(A)$ and $\det(B)$ coincide, and $\det(A)=\det(B)$.
\end{lemma}

\begin{proof} To each
$\sigma=\left(\begin{smallmatrix}{1} &{2} &{\ldots } &{n} \\{i_{1} } &{i_{2} } &{\ldots } &{i_{n} } \end{smallmatrix}\right)\in S_n$   we associate  the following terms
\[
\begin{split}
\gamma_{\sigma}&=(-1)^{sign(\sigma)}{\lambda}_{1,i_1} h_{1,i_1} {\lambda}_{2,i_2} h_{2,i_2}\cdots {\lambda}_{n,i_n} h_{n,i_n}\\&=(-1)^{sign(\sigma)} (h_{1,i_1}  h_{2,i_2}\cdots h_{n,i_n})
({\lambda}_{1,i_1} {\lambda}_{2,i_2} \cdots {\lambda}_{n,i_n});\\
\mu_{\sigma}&=(-1)^{sign(\sigma)}{\lambda}_{i_1,1} h_{1,i_1} {\lambda}_{i_2,2} h_{2,i_2}\cdots {\lambda}_{i_n,n} h_{n,i_n}\\&=(-1)^{sign(\sigma)} (h_{1,i_1}  h_{2,i_2}\cdots h_{n,i_n})
({\lambda}_{i_1,1} {\lambda}_{i_2,2} \cdots {\lambda}_{i_n,n}),
\end{split}
\]
from the decompositions of  $\det(A)$ and $\det(B)$, respectively,  in which
\[
\lambda_{p,i_p}= \begin{cases}
a_{p,i_p},  & \quad\text{ if } p-i_p=1;\\
b_{p,i_p}, & \quad\text{ if } p-i_p>1;\\
1,  &\quad\text{ if } p-i_p  <1.
\end{cases}
\]
It is easy to see that the product ${\lambda}_{1,i_1} {\lambda}_{2,i_2} \cdots {\lambda}_{n,i_n}$ which arises in the decomposition of  the term $\gamma_{\sigma}$ in   $\det(A)$,  corresponds to the following  permutation
$\sigma=\left(\begin{smallmatrix}{1} &{2} &{\ldots } &{n} \\{i_{1} } &{i_{2} } &{\ldots } &{i_{n} } \end{smallmatrix}\right)\in S_n$.
Similarly, the product ${\lambda}_{i_1,1} {\lambda}_{i_2,2} \cdots {\lambda}_{i_n,n}$ which arises in the decomposition of  the term $\mu_{\sigma}$ in   $\det(B)$,  corresponds to the following  permutation $\sigma^{-1}=\left(\begin{smallmatrix} {i_{1} } &{i_{2} } &{\ldots } &{i_{n} } \\{1} &{2} &{\ldots } &{n} \end{smallmatrix}\right)\in S_n$. By virtue of Lemma \ref{Lemma:3},
$$
{\lambda}_{1,i_1} {\lambda}_{2,i_2} \cdots {\lambda}_{n,i_n}=
{\lambda}_{i_1,1} {\lambda}_{i_2,2} \cdots {\lambda}_{i_n,n}.
$$
Consequently, $\gamma_{\sigma}=\mu_{\sigma}$ for all  $\sigma\in S_n$, so
$\det(A)=\det(B)$.
\end{proof}

\section{Main result and its proof}
To simplify notation, in what follows we omit the bar from above when referring to the elements of the ring $R_{m}$.

\begin{theorem}\label{Th:1}
Let $R$ be  a commutative B\'ezout domain (with the property $1\not=0$) of stable range $1.5$.
Let $U(R)$ be  the group of units of   $R$. For each $m\in R\setminus \{U(R), 0\}$ we denote  the factor ring $R_{m} =R/mR$.

Let ${\bf{G}}_{\Phi }\leq {\rm GL}_n(R_m)$ be the Zelisko group  of the following matrix:
\begin{equation}\label{EQV:5}
\Phi: ={\rm diag}(1,  \ldots, 1, \varphi_{t}, \varphi_{t+1},  \ldots, \varphi_{k}, 0,  \ldots, 0)\in R_{m}^{n\times n}
\end{equation}
in which \quad   $\varphi_{t} | \varphi_{t+1} |   \cdots |  \varphi_{k} \ne 0$, \quad $\varphi_{t} \notin U(R_m)$,\quad  $1\leq t$ \quad and\quad   $k \leq n$.

Then the group  ${\bf{G}}_{\Phi }$ consists of all  invertible  matrices $H$ of the form:
\begin{itemize}
\item[\rm{(i)}] $H=
\left(\begin{matrix}
H_{11} & H_{12} & H_{13}   \\H_{21} & H_{22} & H_{23}   \\{\bf 0}  & H_{32} & H_{33}
\end{matrix}
\right)$ for $1<t$ and $k<n$;
%\end{equation}

\item[\rm{(ii)}] $H=
\left(\begin{matrix}
 H_{22} & H_{23}   \\
 H_{32} & H_{33}
\end{matrix}
\right)$ for $1=t$ and $k< n$;

\item[\rm{(iii)}] $H=H_{22}$ for  $1=t$ and $k= n$;

\item[\rm{(iv)}] $H=
\left(\begin{matrix}
 H_{11} & H_{12}   \\
 H_{21} & H_{22}
\end{matrix}
\right)$ for $1<t$ and $k=n$;

\item[\rm{(v)}] $H=
\left(\begin{matrix}
 M_{11} & M_{12}   \\
 {\bf  0}   & M_{22}
\end{matrix}
\right)$, in which  $M_{11}\in {\rm GL}_s(R)$,  $M_{22}\in {\rm GL}_{n-s}(R)$,   $1\leq s< n$ and $($see \eqref{EQV:5}$)$
the matrix $\Phi$ has the following form:
\[
\Phi: ={\rm diag}(\underbrace{1,  \ldots, 1}_s, \underbrace{0,  \ldots, 0}_{n-s})\in R_{m}^{n\times n}.
\]
\end{itemize}
In all of the above cases {\rm (i)-(iv)}, we have
\begin{align}
%\begin{split}
H_{21}&=\left[
\begin{matrix}
\varphi_{t}{h_{t1} } & \varphi_{t}{h_{t2} } &{\cdots} &  \varphi_{t}{h_{t,k-1} } \\\varphi_{t+1}{h_{{t+1},1} } & \varphi_{t+1}{h_{t+1,2} } &{\cdots} &  \varphi_{t+1}{h_{t+1,k-1} }  \\{\cdots} &{\cdots} &{\cdots}  &{\cdots} \\\varphi_{k}{h_{k1} } & \varphi_{k}{h_{k2} } &{\cdots} &  \varphi_{k}{h_{k,k-1} } \end{matrix}  \right];\label{EQV:6}\\
H_{22}&=\left[
\begin{matrix}
{h_{tt} } &{h_{t,t+1} } &{\cdots} &{h_{t,\; k-1} } &{h_{tk} } \\\tifracccc{\varphi_{t+1} }{\varphi_{t}} h_{t+1,t} & h_{t+1,t+1} &{\ldots} &{h_{t+1,  k-1} } &{h_{t+1,k} } \\{\cdots} &{\cdots} &{\cdots} &{\cdots} &{\cdots} \\
\tifracc{\varphi_{k}}{\varphi_{t}} h_{kt}  & \tifracccc{\varphi_{k}}{\varphi_{t+1}} h_{k,t+1}  &{\ldots} &
\tifracccc{\varphi_{k} }{\varphi_{k-1}} h_{k, k-1}  &{h_{kk} } \end{matrix}
   \right];\label{EQV:7}\\
H_{32}&=
\left[\begin{matrix}
\alpha_{t}{h_{k+1,t} } & \alpha_{t+1}{h_{k+1,t+1} } &{\cdots} &  \alpha_{k}{h_{k+1,k} } \\\alpha_{t}{h_{{k+2},t} } & \alpha_{t+1}{h_{k+2,t+1} } &{\cdots} &  \alpha_{k}{h_{k+2,k} }  \\
{\vdots} &{\vdots} &{\vdots}  &{\vdots} \\\alpha_{t}{h_{nt} } & \alpha_{t+1}{h_{n,t+1} } &{\cdots} &  \alpha_{k}{h_{nk} }
 \end{matrix}
   \right],\label{EQV:8}
%   \end{split}
\end{align}
and $H_{11}, H_{12}, H_{13}, H_{23}, H_{33}$ are some matrices of corresponding sizes.
\end{theorem}

\begin{proof} {\rm (i)}
Let  ${\Phi }$ be a  matrix of the form \eqref{EQV:5} in which $1< t$ and $k<n$. A matrix $H:=[ p_{ij}] \in {\bf{G}}_{\Phi }$ we represent  in the form of a block matrix, in which $H_{11}$ of size $t\times t$, $H_{22}$ of size $(k-t+1)\times (k-t+1)$ and $H_{33}$ of size $(n-k)\times (n-k)$, respectively. There exists $S=[ s_{ij}]\in{\rm GL}_n(R_{m})$ such that $H\Phi =\Phi S$
by definition of the Zelisko group. This means that
\begin{equation}\label{EQV:9}
\varphi_{j} p_{ij} =\varphi_{i} s_{ij}.
\end{equation}
If $j\geq i$, then we  put  $s_{ij}:= \tifracc{\varphi_{j} }{\varphi_{i}}p_{ij}$, such that \eqref{EQV:9} holds for arbitrary value of $p_{ij}$.
It follows that  no restrictions are imposed on the elements $p_{ij}$ from the blocks
$ H_{12}, H_{13}$ and  $H_{23}$, because they lie above the main diagonal of $H$.

Considering that the first $t-1$ diagonal elements of the matrix $\Phi$ are $1$ and the last $n -k$ of its diagonal elements are zeros (see \eqref{EQV:5}), we get that no restrictions are imposed on the elements $p_{ij}$ from the blocks
$ H_{11}$ and  $H_{33}$.

If  $j=1, \ldots,  t-1$ and   $i=t,  \ldots, k$ (see \eqref{EQV:5}), then  \eqref{EQV:9} rewrites as
$p_{ij} =\varphi_{i} s_{ij}$.  It follows that $H_{21}$  has the form (\ref{EQV:6}).

Consider the block  $H_{22}$. Taking into account our definition of the element $\tifracc{\varphi_{i}}{\varphi_{j}}$ (see after the proof of Lemma \ref{L:1}) and the proof of \cite[Theorem 2]{Bovdi_Shchedryk_II}, we obtain that $H_{22}$  has the form  (\ref{EQV:7}).

Consider blocks  $H_{31}$ and $H_{32}$. Using the fact that  the first $t-1$ elements of $\Phi$ are $1$ and the last $n-k$ elements of $\Phi$ are $0$, the equation \eqref{EQV:9} rewrites as  $\varphi_{j}p_{ij} = 0$.  This means that
$p_{ij} \in Ann (\varphi_{j}) =\alpha_{j}R_m$, so  $p_{ij} =\alpha_{j} h_{ij}$. It immediately follows
that    $H_{31}={\bf 0}$ and $H_{32}$   has  form \eqref{EQV:8}.

The proofs of  {\rm (ii)-(v)}  are  particular cases of the previous one.

\smallskip

$\Leftarrow$. {\rm (i)}  Consider a block  matrix $S=
\left[\begin{smallmatrix}
S_{11} & S_{12} &{\bf 0}    \\S_{21} & S_{22} & S_{23}   \\S_{31}  & S_{32} & S_{33}
\end{smallmatrix}\right],
$
in which each block $S_{ij}$ has the same size as the size of corresponding block $H_{ij}$ in $H$  and  set
\[
\begin{split}
S_{12}&:=\left[
\begin{matrix}
\varphi_{t}{h_{1t} } & \varphi_{t+1}{h_{1,t+1} } &{\cdots} &  \varphi_{k}{h_{1k} } \\\varphi_{t}h_{2t } & \varphi_{t+1}{h_{2,t+1} } &{\cdots} &  \varphi_{k}{h_{2k} }  \\{\cdots} &{\cdots} &{\cdots}  &{\cdots} \\\varphi_{t}{h_{t-1,t} } & \varphi_{t+1}{h_{t-1,t+1} } &{\cdots} &  \varphi_{k}{h_{t-1,k} }
 \end{matrix}
   \right],
\smallskip
   \\
   S_{22}&:=\left[
\begin{matrix}
{h_{tt} } & \tifracccc{\varphi_{t+1} }{\varphi_{t}}{h_{t,t+1} } &{\cdots} & \tifracccc{\varphi_{k-1} }{\varphi_{t}}{h_{t,\; k-1} } & \tifracc{\varphi_{k}}{\varphi_{t}}{h_{tk} } \\{\cdots} &{\cdots} &{\cdots} &{\cdots} &{\cdots} \\h_{k-1,t} & h_{k-1,t+1} &{\ldots} &{h_{k-1,  k-1} } & \tifracccc{\varphi_{k} }{\varphi_{k-1}}{h_{k-1,k} } \\h_{kt}  &  h_{k,t+1}  &{\ldots} &
h_{k, k-1}  &{h_{kk} } \end{matrix}
   \right],\\
S_{23}&:=\left[
\begin{matrix}
\alpha_{t}{h_{t,k+1} } & \alpha_{t}{h_{t,k+2} } &{\cdots} &  \alpha_{t}{h_{tn} } \\\alpha_{t+1}{h_{t+1,{k+1}} } & \alpha_{t+1}{h_{t+1,k+2} } &{\cdots} &  \alpha_{t+1}{h_{t+1,n} }  \\{\cdots} &{\cdots} &{\cdots}  &{\cdots} \\\alpha_{k}{h_{k,k+1} } & \alpha_{k}{h_{k,k+2} } &{\cdots} &  \alpha_{k}{h_{kn} }
 \end{matrix}
   \right].
\end{split}
\]
It is easy to check that $H\Phi=\Phi S$ for arbitrary $S_{11}, S_{21}, S_{31}, S_{32}$ and  $S_{33}$.

Let us    prove that $S$ is invertible. Consider the following four cases:

Case 1. Let  $p_{ij}\in H_{21}$. Using \eqref{EQV:2} we have
\[
\begin{split}
\textstyle
p_{ij}=\varphi_{i} h_{ij}&= \varphi_{t}\tifracc{\varphi_{i}}{\varphi_{t}} h_{ij}\\
&=
\varphi_{t } \;  \tifracccc{\varphi_{t+1}}{\varphi_{t}} \tifracccc{{\varphi}_{t+2} }{{\varphi}_{t+1}} \cdots
\tifracccc{{\varphi}_{i-1} }{{\varphi}_{i-2}}  \tifracccc{{\varphi}_{i} }{{\varphi}_{i-1}}h_{ij}.
\end{split}
\]

Case 2. Let  $ p_{ij}\in H_{22}$ for $i>j$. Using \eqref{EQV:2} we obtain that
\[
p_{ij}=\tifracc{\varphi_{i}}{\varphi_{j}}h_{ij}=\tifracccc{{{\varphi}}_{j+1}}{{{\varphi}}_j}\tifracccc{{\varphi}_{j+2} }{{\varphi}_{j+1}} \cdots
\tifracccc{{\varphi}_{i-1} }{{\varphi}_{i-2}}  \tifracccc{{\varphi}_{i} }{{\varphi}_{i-1}}h_{ij}.
\]

Case 3. Let  $p_{ij}\in H_{32}$. Using  Corollary \ref{Co:1} and \eqref{EQV:2} we get that
\[
\begin{split}
\textstyle
p_{ij}=\alpha_{j} h_{ij}& =\tifracc{\varphi_{k}}{\varphi_{j}}\alpha_kh_{ij}=\\
&=
\tifracccc{\varphi_{j+1}}{\varphi_{j}} \tifracccc{{\varphi}_{j+2} }{{\varphi}_{j+1}} \cdots
\tifracccc{{\varphi}_{k-1} }{{\varphi}_{k-2}}  \tifracccc{{\varphi}_{k} }{{\varphi}_{k-1}}
\alpha_{k}h_{ij}.
\end{split}
\]

Case 4. Finally, let $ p_{ij}\in H_{31}\equiv {\bf 0}$. It is easy to check that
\[
\begin{split}
\textstyle
p_{ij}=0={\varphi}_{k}\alpha_{k}&=\textstyle\varphi_{t } \tifracc{{\varphi}_{k} }{{\varphi}_{t}}\alpha_{k}\\
&=\varphi_{t } \;  \tifracccc{\varphi_{t+1}}{\varphi_{t}} \tifracccc{{\varphi}_{t+2} }{{\varphi}_{t+1}} \cdots
\tifracccc{{\varphi}_{k-1} }{{\varphi}_{k-2}}  \tifracccc{{\varphi}_{k} }{{\varphi}_{k-1}}\alpha_{k}.
\end{split}
\]
Consequently, all elements of $H$, which lie below the main diagonal of $H$,  satisfy  Lemma \ref{Le:4}, so
$S$  is invertible.

The proofs of  {\rm (ii)-(v)}  are  particular cases of the previous one.
\end{proof}


\begin{thebibliography}{10}

\bibitem{Anderson_Juett}
D.~D. Anderson and J.~R. Juett.
\newblock Stable range and almost stable range.
\newblock {\em J. Pure Appl. Algebra}, 216(10):2094--2097, 2012.

\bibitem{Bass}
H.~Bass.
\newblock {$K$}-theory and stable algebra.
\newblock {\em Inst. Hautes \'{E}tudes Sci. Publ. Math.}, (22):5--60, 1964.

\bibitem{Bovdi_Shchedryk_I}
V.~A. Bovdi and V.~P. Shchedryk.
\newblock Commutative {B}ezout domains of stable range 1.5.
\newblock {\em Linear Algebra Appl.}, 568:127--134, 2019.

\bibitem{Bovdi_Shchedryk_II}
V.~A. Bovdi and V.~P. Shchedryk.
\newblock Generating solutions of a linear equation and structure of elements
  of the {Z}elisko group.
\newblock {\em Linear Algebra Appl.}, 625:55--67, 2021.

\bibitem{Chen_2}
H.~Chen.
\newblock On simple factorization of invertible matrices.
\newblock {\em Linear Multilinear Algebra}, 55(1):81--92, 2007.

\bibitem{Chen}
H.~Chen.
\newblock {\em Rings related to stable range conditions}, volume~11 of {\em
  Series in Algebra}.
\newblock World Scientific Publishing Co. Pte. Ltd., Hackensack, NJ, 2011.

\bibitem{Kaz2}
P.~S. Kaz\={\i}m\={\i}rski\u{\i}.
\newblock A solution to the problem of separating a regular factor from a
  matrix polynomial.
\newblock {\em Ukrain. Mat. \v{Z}.}, 32:483--498, 1980.

\bibitem{McGovern}
W.~W. McGovern.
\newblock B\'{e}zout rings with almost stable range 1.
\newblock {\em J. Pure Appl. Algebra}, 212(2):340--348, 2008.

\bibitem{Petechuk_Petechuk}
V.~M. Petechuk and J.~V. Petechuk.
\newblock Isomorphisms of matrix groups over commutative rings.
\newblock {\em Acta Sci. Math. (Szeged)}, 83(1-2):113--123, 2017.

\bibitem{Shch_Mon}
V.~Shchedryk.
\newblock {\em Arithmetic of matrices over rings}.
\newblock Graduate Texts in Mathematics,
Akademperiodyka, Kiev,  2021.  https://doi.org/10.15407/akademperiodika.430.278
%\newline
 http://www.iapmm.lviv.ua/14/index.htm.

\bibitem{Shchedryk_3}
V.~P. Shchedryk.
\newblock Some properties of primitive matrices over {B}ezout {$B$}-domain.
\newblock {\em Algebra Discrete Math.}, (2):46--57, 2005.

\bibitem{Shch_StablRange}
V.~P. Shchedryk.
\newblock Bezout rings of stable range 1.5.
\newblock {\em Ukrainian Math. J.}, 67(6):960--974, 2015.
\newblock Translation of Ukra\"{\i}n. Mat. Zh. {{\bf{6}}7} (2015), no. 6,
  849--860.

\bibitem{Shch_DecompGroup}
V.~P. Shchedryk.
\newblock Bezout rings of stable rank 1.5 and the decomposition of a complete
  linear group into products of its subgroups.
\newblock {\em Ukra\"{\i}n. Mat. Zh.}, 69(1):113--120, 2017.

\bibitem{Zel}
V.~R. Zel\={\i}sko.
\newblock Construction of a class of invertible matrices.
\newblock {\em Mat. Metody i Fiz.-Mekh. Polya}, 12:14--21, 120, 1980.

\end{thebibliography}
\end{document}